\documentclass[12pt,reqno]{amsart}
\usepackage[greek,english]{babel}
\usepackage{amsthm}
\usepackage{amsbsy}
\usepackage{amsfonts}
\usepackage{graphicx}
\usepackage{hyperref}
\hypersetup{colorlinks=true, citecolor=blue, linkcolor=red}

{ \theoremstyle{plain}

}

{ \theoremstyle{definition}
\newtheorem{thm}{Theorem}

\newtheorem{rem}{Remark}

}

\begin{document}

\title[On a partially overdetermined problem in a cone]{On a partially overdetermined problem in a cone}



\author{Christos Sourdis}
\address{Department of Mathematics,  University of
	Athens.}
\email{christos.sourdis@unito.it}




\maketitle

\begin{abstract}
We prove a rigidity result for Serrin's overdetermined problem in a cone that is contained in a half-space in arbitrary dimensions. In the special case where the cone is an epigraph, this result was shown previously  in low dimensions with a different approach. \end{abstract}

\section*{Introduction and proof of the main result}

In Corollary 9 of their paper \cite{FarinaValdAMS},  Farina and Valdinoci considered the following partially overdetermined problem in the cone
\[
\Omega=\left\{x=(x',x_n)\in \mathbb{R}^n,\ n\geq 2\ :\ x_n>\alpha |x'| \right\}\ \textrm{with} \ \alpha \geq 0:
\]

 \begin{equation}\label{eqEq}
 \left\{
 \begin{array}{c}
   \Delta u+f(u)=0,\ u>0\ \textrm{in}\ \Omega; \\
    \\
   u=0,\ \partial_\nu u=c\ \textrm{on}\ \partial \Omega \setminus \{0\},
 \end{array}
 \right.
 \end{equation}
 with
 \begin{equation}\label{eqRegU}
 u\in C^2(\Omega)\cap C^1\left(\bar{\Omega}\setminus \{ 0\} \right)\cap W^{1,\infty}(\Omega),
 \end{equation}
where $\nu$ denotes the exterior unitary normal vector on $\partial \Omega \setminus \{ 0\}$, $c\in \mathbb{R}$ is any fixed constant, and $f\in C^1(\mathbb{R})$.
 It is shown therein that \emph{ $\alpha=0$, provided that $n\leq 3$}.
 It is worth mentioning that the above problem was initially studied  in \cite{FVRevist}, motivated by a question of  Vazquez.
Their result represents an extension of the famous \emph{Serrin's problem} \cite{Serrin} in unbounded epigraphs, in the spirit of \cite{BCN}. We point out that the characterization 'partially overdetermined' comes from the fact that the overdetermined boundary conditions are not prescribed in the entire $\partial \Omega$.

In this short note, using a completely different approach, we prove a generalization of this result in any dimension.
In fact, as will be apparent, our arguments go through with a bit weaker regularity assumptions on $f$. Moreover, $W^{1,\infty}_{loc}(\Omega)$ in (\ref{eqRegU}) suffices for our purposes (see also (\ref{eqbelow}) below). Our approach is greatly motivated from the study of the regularity properties of free boundaries in one-phase  and obstacle-type problems, and hinges on the fact that $\lambda \Omega \equiv \Omega$ for any cone $\Omega$ with vertex at the origin and $\lambda>0$.
\begin{thm}
\emph{Let $\Omega$
be an open cone in $\mathbb{R}^n$, $n\geq 2$, that is contained in the half-space $\{x_n> 0\}$ and such that $\partial \Omega \setminus \{0\}$  has $C^{1,\beta}$ regularity for some $\beta>0$.
Moreover, let $c,\ f$ be as above and let $u$ satisfy (\ref{eqEq}) and (\ref{eqRegU}). Then, the cone $\Omega$ coincides with the half-space  $\{x_n> 0\}$.}\end{thm}
\begin{proof}
We first consider the case
\[
c\neq 0\ \textrm{i.e.}\ c<0.
\]

Motivated from the study of one-phase free boundary problems \cite{alt}, for $r>0$ small, we consider the following \emph{blow-up} of $u$:
\[
u_r(y)=\frac{u(ry)}{r},\ \ y\in \Omega.
\]
We readily find that
\[
 \left\{
 \begin{array}{c}
   \Delta u_r+rf(ru_r)=0,\ u_r>0\ \textrm{in}\ \Omega; \\
    \\
   u_r=0,\ \partial_\nu u_r=c\ \textrm{on}\ \partial \Omega \setminus \{0\}.
 \end{array}
 \right.
 \]

 By virtue of (\ref{eqRegU}), which implies that
  \begin{equation}\label{eqbelow}
  u(x)\leq C|x| \ \textrm{for some}\ C>0\ \textrm{near the origin (see also \cite[Thm. 4.1]{heino})},
  \end{equation}
  and standard elliptic estimates \cite[Ch. 9-10]{jost} (interior $W^{2,p}$ estimates and boundary $C^{1,\beta}$ Schauder estimates, keeping in mind that the cone becomes flatter and flatter at infinity), along a sequence $r_j\to 0$, $u_{r_j}$ converges in $C^1_{loc}\left(\bar{\Omega}\setminus \{ 0\}\right)$
 to some blow-up limit $u_0\in C^2(\Omega)\cap C\left(\bar{\Omega}\right)\cap C^1\left(\bar{\Omega}\setminus \{ 0\} \right)$ which solves
\begin{equation}\label{eqEqBU}
 \left\{
 \begin{array}{l}
   \Delta u_0 =0,\ u_0>0\ \textrm{in}\ \Omega; \\
    \\
   u_0=0,\ \partial_\nu u_0=c\ \textrm{on}\ \partial \Omega \setminus \{0\}.
 \end{array}
 \right.
 \end{equation} We point out that we got a nontrivial limit $u_0$   because of the assumption that $c<0$.

By a result of \cite{ancona}, all positive harmonic functions in $\Omega$ that vanish on $\partial \Omega$ must be homogeneous, i.e,
\[
u_0(y)=|y|^\gamma \Phi_0\left(\frac{y}{|y|} \right),
\]
for some $\gamma>0$ and $\Phi_0 \in C\left(\overline{\mathbb{S}^{n-1}\cap \Omega} \right)\cap C^2\left(\mathbb{S}^{n-1}\cap \Omega \right)$ which vanishes on
$\mathbb{S}^{n-1}\cap \partial \Omega$. Observe that, since $\nabla u_0$ is a homogeneous function of degree $\gamma-1$, in order for the overdetermined boundary conditions in (\ref{eqEqBU}) to be satisfied, we must have
\[
\gamma=1.
\]

So,  $\Phi_0$ is a positive eigenfunction with Dirichlet boundary conditions to the Laplace-Beltrami operator $-\Delta_{\mathbb{S}^{n-1}}$ on $\mathbb{S}^{n-1}\cap \Omega$, corresponding to the eigenvalue $n-1$ (see for instance the proof of \cite[Lem. 2.1]{BCN}). Hence, recalling that $n-1$ is the principal Dirichlet eigenvalue of $-\Delta_{\mathbb{S}^{n-1}}$ on the upper half-sphere $\mathbb{S}^{n-1}_+$, we deduce that \[\left|\mathbb{S}^{n-1}\cap \Omega \right|=\left|\mathbb{S}^{n-1}_+\right|=\frac{1}{2}\left|\mathbb{S}^{n-1}\right|.\]
The above relation, however, is only possible if $\Omega$ coincides with the half-space  $\{x_n> 0\}$ as desired.

It remains to consider the case
\[
c=0.
\]

Firstly, by Hopf's boundary point lemma (at some point on $\partial \Omega \setminus \{0\}$ which belongs to the boundary of a  ball contained  in $\Omega$), we deduce that
\begin{equation}\label{eqf}
f(0)<0.
\end{equation}

It follows readily from (\ref{eqEq}) and (\ref{eqRegU}) that $u$ (extended trivially outside of $\Omega$) satisfies, in the weak sense, the following problem:
\begin{equation}\label{heavy}
\Delta u=-H(u)f(u)\ \ \textrm{in}\ \mathbb{R}^n,
\end{equation}
where $H$ stands for the usual Heaviside function. We just point out that near the origin one uses that
\[
\lim_{\rho \to 0}\int_{B_\rho(0)}^{}\nabla u \nabla \varphi dx=0,
\]
for any $\varphi \in C_0^\infty(\mathbb{R}^n)$, which holds since $u\in W^{1,\infty}$ (see also \cite[Thm. 1.3]{cabre1} for a related argument).

Using (\ref{eqf}) and the assumed regularity on $f$, $u$, it follows from \cite[Ch. 2]{petro} that
\begin{equation}\label{eqC1,1}
u\in C^{1,1}\left(B_\delta(0)\right)\ \ \textrm{for some small}\ \delta>0,
\end{equation}
(see also \cite{shagho} for a more general approach).

This time, as in the study of free boundary problems of obstacle type \cite{petro}, for $r>0$ small, we consider the following {blow-up} of $u$:
\[
u_r(y)=\frac{u(ry)}{r^2},\ \ y\in \Omega.
\]
We readily find that
\[
 \left\{
 \begin{array}{c}
   \Delta u_r+f(r^2u_r)=0,\ u_r>0\ \textrm{in}\ \Omega; \\
    \\
   u_r=0,\ \partial_\nu u_r=0\ \textrm{on}\ \partial \Omega \setminus \{0\}.
 \end{array}
 \right.
 \]
By virtue of (\ref{eqC1,1}) and standard elliptic estimates, along a sequence $r_j\to 0$, $u_{r_j}$ converges in $C^1_{loc}(\mathbb{R}^n)$
 to some blow-up limit $u_0\in C^2(\Omega)\cap C^{1,1}\left(\mathbb{R}^n\right)$ (globally) which satisfies
\begin{equation}\label{eqEqBU}
 \left\{
 \begin{array}{l}
   \Delta u_0 =-f(0),\ u_0>0\ \textrm{in}\ \Omega; \\
    \\
   u_0=0,\ \partial_\nu u_0=0\ \textrm{on}\ \partial \Omega \setminus \{0\}.
 \end{array}
 \right.
 \end{equation}
 We point out that we got a nontrivial limit $u_0$   because of (\ref{eqf}).

Let $e$ be an arbitrary direction in $\mathbb{R}^n$ and let
\[
v=\partial_e u_0.
\]
Then, the function $v$ is harmonic in the cone $\Omega$, vanishes on its boundary, and is globally Lipschitz continuous.
To conclude, let us suppose, to the contrary, that the cone $\Omega$ does not coincide with the half-space $\{x_n>0 \}$.
Then, since $v$ has at most linear growth, we get from Lemma 2.1 in \cite{BCN} and the comments after it that
\[
v\equiv 0.
\]
On the other hand, recalling that the direction $e$ is arbitrary, this  contradicts the fact that $u_0$ is nontrivial.

The proof of the theorem is complete.
\end{proof}

\begin{rem}
\emph{The initial value problem
\[
\ddot{u}=-f(u);\ u(0)=0,\ \dot{u}(0)=c, \
\]
has a unique local solution $U$  (not necessarily positive) with maximal interval of existence $[0,T)$, $T\leq +\infty$. By Hopf's boundary lemma, applied to the difference $u(x)-U(x_n)$, we infer that $u$ coincides with $U$ for $0\leq x_n<T$. In particular, the partial derivatives $\partial_{x_i}u$, $i<n$ are identically zero in this strip, and thus in the entire half-space $\{x_n>0 \}$ (by the unique continuation principle applied in the linear equation that each one satisfies). Consequently, $u$ depends only on $x_n$.}
\end{rem}

\begin{rem}
\emph{It is well known that, in any dimension $n\geq 3$, there exists an $\alpha_n>0$ such that the cone described by $|x_n|<\alpha_n |x'|$ supports a one-homogenous solution to (\ref{eqEq}) with $f \equiv 0$, see \cite{alt}.}
\end{rem}

\end{document}